\documentclass[11pt]{article}
\usepackage{amsmath, amsfonts, amssymb, amsthm}

\newtheorem{thm}{Theorem}
\newtheorem{prop}[thm]{Proposition}
\newtheorem{lem}[thm]{Lemma}

\newtheorem{rem}[thm]{Remark}
\newtheorem{df}[thm]{Definition}
\newtheorem{ex}[thm]{Example}

\renewcommand{\epsilon}{\varepsilon}
\renewcommand{\phi}{\varphi}

\newcommand{\inv}{\operatorname{Inv}}

\begin{document}

\title{Geometric Properties of Conformal Transformations on $\mathbb{R}^{p,q}$}

\author{Matvei Libine\footnote{Department of Mathematics, Indiana University,
Rawles Hall, 831 East 3rd St, Bloomington, IN 47405}
 and Surya Raghavendran\footnote{Undergraduate student at
the University of Texas, Austin, TX 78712}}

\maketitle

\begin{abstract}
We show that conformal transformations on the
generalized Minkowski space $\mathbb{R}^{p,q}$ map hyperboloids and
affine hyperplanes into hyperboloids and affine hyperplanes.
We also show that this action on hyperboloids and affine hyperplanes is
transitive when $p$ or $q$ is $0$,
and that this action has exactly three orbits if $p, q \ne 0$.
Then we extend these results to hyperboloids and affine planes of
arbitrary dimension.
These properties generalize the well-known properties of M\"{o}bius
(or fractional linear) transformations on the complex plane
$\mathbb{C}$.
\end{abstract}

{\bf Keywords:}
conformal transformations on $\mathbb{R}^{p,q}$,
fractional linear transformations,
M\"obius transformations, M\"obius geometry.

\section{Introduction}
In this article we investigate geometric properties of conformal
transformations on the generalized Minkowski space $\mathbb{R}^{p,q}$.
Recall that $\mathbb{R}^{p,q}$ is the vector space $\mathbb{R}^{p+q}$
equipped with indefinite quadratic form
$$
Q(\overrightarrow{x})=(x_1)^2+\cdots+(x_p)^2-(x_{p+1})^2-\cdots-(x_{p+q})^2.
$$
The indefinite orthogonal group $O(p+1,q+1)$ acts on $\mathbb{R}^{p,q}$
by conformal transformations.
If $p+q>2$, in a certain sense, these are all possible conformal
transformations of $\mathbb{R}^{p,q}$.

When $p=2$ and $q=0$, this action reduces to the
well-familiar action of $SL(2,\mathbb{C})$ on the complex plane
$\mathbb{C} \simeq \mathbb{R}^{2,0}$ by M\"obius (or fractional linear)
transformations.
We are particularly interested in the facts that M\"obius transformations
on $\mathbb{C}$ map circles and lines into circles and lines and that
this action of $SL(2,\mathbb{C})$ on circles and lines is transitive.

In this paper we find analogues of these properties for the conformal action of
$O(p+1,q+1)$ on $\mathbb{R}^{p,q}$ with arbitrary $p$ and $q$.
A special case of quaternions $\mathbb{H} \simeq \mathbb{R}^{4,0}$
was worked out in \cite{BG,GL}.
Another previously known case is $\mathbb{R}^{1,1}$ \cite{YAG}.
In this case circles get replaced with hyperbolas.
These cases suggest that in the general case, where the quadratic form $Q$
is not necessarily positive definite, circles should be replaced with
hyperboloids with quadratic term proportional to $Q$.
Example \ref{R^{1,1}-ex} shows that it is necessary to work in the conformal
compactification $N^{p,q}$ of $\mathbb{R}^{p,q}$ and that it is also necessary
to find suitable analogues of hyperboloids and hyperplanes in $N^{p,q}$;
we call these analogues {\em conformal quadratic hypersurfaces}.
The key step is to define the conformal quadratic hypersurfaces
as intersections of $N^{p,q} \subset \mathbb{R}P^{p+q+1}$ with hyperplanes
in the real projective space $\mathbb{R}P^{p+q+1}$ (Definition \ref{def}).
This definition is very easy to work with, and we immediately conclude that
the group $O(p+1,q+1)$ acting on $N^{p,q}$ maps these conformal quadratic
hypersurfaces into themselves.
Moreover, this action of $O(p+1,q+1)$ on the conformal quadratic hypersurfaces
is transitive if $p$ or $q$ is $0$, this action has exactly three orbits when
$p,q \ne 0$, and we give a simple necessary and sufficient condition for two
such hypersurfaces to be in the same orbit (Theorem \ref{main-thm}).
Then we extend this result to lower-dimensional hyperboloids and affine planes
(Theorems \ref{main-thm2} and \ref{main-thm3}).

Theorems \ref{main-thm}, \ref{main-thm2} and \ref{main-thm3} can be regarded
as Clifford algebra analogues of the case of complex plane $\mathbb{C}$.
In particular, suppose that $p=0$, i.e. that the quadratic form $Q$ is negative
definite. Then $N^{0,q}$ is the one-point compactification of $\mathbb{R}^q$.
Let ${\cal C}_q$ denote the Clifford algebra associated to $(\mathbb{R}^q, Q)$.
In this case, the conformal transformations on $N^{0,q}$ produced by
$O(1,q+1)$ can be expressed as fractional linear transformations
$$
\overrightarrow{x} \mapsto
(a \overrightarrow{x} + b) (c \overrightarrow{x} +d)^{-1},
$$
where $\bigl(\begin{smallmatrix} a & b \\ c & d \end{smallmatrix}\bigr)$
are the so-called Vahlen matrices with entries $a,b,c,d$ in the Clifford group
$$
\Gamma_q = \{ y \in {\cal C}_q ;\: y \text{ is invertible in } {\cal C}_q \}
$$
(see, for example, \cite{V,A2}).
Theorem \ref{main-thm2} asserts that these transformations
map $d$-dimensional spheres and affine planes in $\mathbb{R}^q$ into
spheres and affine planes of the same dimension, where $d=1,\dots, q-1$.
Furthermore, the group $O(1,q+1)$ acts transitively on the set of
$d$-dimensional spheres and affine planes in $\mathbb{R}^q$.

We would like to mention a connection between this geometric work,
representation theory and solutions of the ultrahyperbolic wave equation. Let
$$
\square_{p,q} = \frac{\partial^2}{(\partial x_1)^2} + \dots
+ \frac{\partial^2}{(\partial x_p)^2} - \frac{\partial^2}{(\partial x_{p+1})^2}
- \dots - \frac{\partial^2}{(\partial x_{p+q})^2}
$$
be the ultrahyperbolic wave operator with symbol $Q(\overrightarrow{x})$.
The conformal action of $O(p+1,q+1)$ on $\mathbb{R}^{p,q}$ lifts to a
meromorphic linear action on the space of solutions of the 
wave equation $\square_{p,q} \phi=0$.
When $p+q$ is even and $\ge 6$, the groups $O(p+1,q+1)$'s have certain
unitary representations called {\em minimal representations}.
These representations have been studied by many authors, and \cite{KO}
is a particularly comprehensive work on this subject.
In this work the authors -- among other things -- use the linear action of
$O(p+1,q+1)$ on the space of solutions of $\square_{p,q} \phi=0$ to construct
a concrete realization of the minimal representation of $O(p+1,q+1)$, then
in Theorem 6.2 in Part III of \cite{KO} they describe the
$O(p+1,q+1)$-invariant inner product, which is expressed by a certain integral
over a conformal quadratic hypersurface in $N^{p,q}$ (with non-zero sign).
The independence of the choice of a conformal quadratic hypersurface is
essentially a manifestation of the unitarity of the minimal representation
and the fact that any conformal quadratic hypersurface can be mapped into any
other conformal quadratic hypersurface of the same sign by an element of
$O(p+1,q+1)$ (Theorem \ref{main-thm}).
As it was discussed in Section 6 of Part III of \cite{KO}, the property that
this inner product is independent of the choice of a conformal quadratic
hypersurface leads to a remarkable ``conservation law'' for the solutions
of the ultrahyperbolic equation $\square_{p,q} \phi=0$.

The paper is organized as follows.
In Section \ref{S2} we fix our notation and review the properties of the
conformal compactification $N^{p,q}$ of $\mathbb{R}^{p,q}$.
In Section \ref{S3} we discuss the action of $O(p+1,q+1)$ on the set
of hyperboloids and hyperplanes in $\mathbb{R}^{p,q}$ and give
Example \ref{R^{1,1}-ex}. This example demonstrates that it is necessary to
work in the conformal compactification $N^{p,q}$ rather than $\mathbb{R}^{p,q}$
and that it is also necessary to find suitable analogues of hyperboloids and
hyperplanes in $N^{p,q}$.
In Section \ref{S4} we find these analogues and give our key
Definition \ref{def} of conformal quadratic hypersurfaces in $N^{p,q}$.
In Section \ref{S5} we prove our main results
(Theorems \ref{main-thm}, \ref{main-thm2} and \ref{main-thm3}).
Let $\hat{\cal H}_d$ denote the set of surfaces in $N^{p,q}$ that serve as
analogues of hyperboloids and affine planes of dimension $d$.
Then Theorems \ref{main-thm}, \ref{main-thm2} and \ref{main-thm3}
state that the group $O(p+1,q+1)$ acting on $N^{p,q}$ preserves each
$\hat{\cal H}_d$ and also provide necessary and sufficient conditions for
two elements in $\hat{\cal H}_d$  to be in the same $O(p+1,q+1)$-orbit.

This paper was written as a part of REU (research experiences for
undergraduates) project at Indiana University during Summer 2014.
We would like to thank Professors Chris Connell and Larry Moss,
and the NSF for providing the organization and funding for the REU program
that made this project possible.
This REU program was supported by the NSF grant DMS-1156515.

\section{The Conformal Compactification of $\mathbb{R}^{p,q}$}  \label{S2}

\subsection{The General Case}

In this section we fix our notation and review the properties of
the conformal compactification $N^{p,q}$ of $\mathbb{R}^{p,q}$ as well as
the action of $O(p+1,q+1)$ on $N^{p,q}$ by conformal transformations.
The reader may wish to refer to, for example, \cite{SCH, KO, K}
for more detailed expositions of the results stated in this section.

Let $p,q = 0,1,2,\dots$ with $p+q \ge 2$, and let $\mathbb{R}^{p,q}$ denote
the vector space $\mathbb{R}^{p+q}$ equipped with indefinite quadratic form
$$
Q(\overrightarrow{x}) = (x_1)^2+\dots+(x_p)^2-(x_{p+1})^2-\dots-(x_{p+q})^2.
$$
We denote by
\begin{align*}
B(\overrightarrow{x},\overrightarrow{y})&=
\frac12 \bigl[ Q(\overrightarrow{x}+\overrightarrow{y})
- Q(\overrightarrow{x}) - Q(\overrightarrow{y}) \bigr] \\
&=x_1y_1+\dots+x_py_p-x_{p+1}y_{p+1}-\dots -x_{p+q}y_{p+q},
\end{align*} 
the unique bilinear form obtained by polarizing $Q$.
Associated with the quadratic form is the indefinite orthogonal group
consisting of all invertible matrices that preserve $Q$
$$
O(p,q)= \bigl\{ M\in GL(p+q,\mathbb{R}) ;\:
Q(M\overrightarrow{x})=Q(\overrightarrow{x})\text{ for all }
\overrightarrow{x} \in \mathbb{R}^{p,q} \bigr\}.
$$

By a {\em conformal transformation} of $\mathbb{R}^{p,q}$, we mean a
smooth mapping $\phi: U \to \mathbb{R}^{p,q}$, where
$U \subset \mathbb{R}^{p,q}$ is a non-empty connected open subset, such that
$$
\phi^*B(\overrightarrow{x},\overrightarrow{y})
= \Omega^2 \cdot B(\overrightarrow{x},\overrightarrow{y})
$$
for some smooth function $\Omega: U \to (0,\infty)$.

Next we introduce a {\em conformal compactification} of $\mathbb{R}^{p,q}$,
denoted by $N^{p,q}$.
For this purpose we consider $\mathbb{R}^{p+1,q+1}$ with quadratic form
\begin{equation}  \label{Q-hat}
\hat Q(\overrightarrow{\xi}) = (\xi_0)^2+(\xi_1)^2+\dots+(\xi_p)^2
-(\xi_{p+1})^2-\dots-(\xi_{p+q})^2-(\xi_{p+q+1})^2.
\end{equation}
By definition, $N^{p,q}$ is a quadric in the real projective space:
$$
N^{p,q} = \bigl \{ [\overrightarrow{\xi}] \in \mathbb{R}P^{p+q+1} ;\:
\hat{Q}(\overrightarrow{\xi})=0 \bigr\}.
$$
The quotient map
$\mathbb{R}^{p+1,q+1} \setminus \{0\} \twoheadrightarrow \mathbb{R}P^{p+q+1}$
restricted to the product of unit spheres
$$
S^p \times S^q = \biggl\{ \overrightarrow{\xi} \in \mathbb{R}^{p+q+2} ;\:
\sum_{j=0}^p (\xi_j)^2 = \sum_{j=p+1}^{p+q+1} (\xi_j)^2 = 1 \biggr\}
\subset \mathbb{R}^{p+1,q+1}
$$
induces a smooth $2$-to-$1$ covering
$$
\pi : S^p\times S^q \to N^{p,q},
$$
which is a local diffeomorphism.
The product $S^p \times S^q$ has a semi-Riemannian metric coming from the
product of metrics --  the first factor having the standard metric from
the embedding  $S^p \subset \mathbb{R}^{p+1}$ as the unit sphere,
and the second factor $S^q$ having the negative of the standard metric.
Then the semi-Riemannian metric on $N^{p,q}$ is defined by declaring
the covering $\pi : S^p\times S^q \to N^{p,q}$ to be a local isometry.
We also see that $N^{p,q}$ is compact, since it is the image of a compact set
$S^p \times S^q$ under a continuous map $\pi$.

The space $\mathbb{R}^{p,q}$ can be embedded into $N^{p,q}$ as follows.
First, we define an embedding into the projective space
\begin{align*}
\iota : \mathbb{R}^{p,q} &\hookrightarrow \mathbb{R}P^{p+q+1} \\ 
\overrightarrow{x}=(x_1,\dots,x_{p+q}) &\mapsto
\iota(\overrightarrow{x}) =\left[ \frac{1-Q(\overrightarrow{x})}2
:x_1:\dots:x_{p+q}:\frac{1+Q(\overrightarrow{x})}2 \right],
\end{align*} 
where the image of $\overrightarrow{x}$ is given in homogeneous coordinates.
Then we observe that the image $\iota(\mathbb{R}^{p,q})$ lies inside $N^{p,q}$.

\begin{prop}
The quadric $N^{p,q}$ is indeed a conformal compactification of
$\mathbb{R}^{p,q}$.
That is, the map $\iota :\mathbb{R}^{p,q}\to N^{p,q}$ is a conformal embedding
and the closure $\overline{\iota(\mathbb{R}^{p,q})}$ is all of $N^{p,q}$.
\end{prop}

Note that when $p$ or $q$ is $0$, $N^{p,q}$ is just the one-point
compactification of $\mathbb{R}^{p+q}$.
When $p,q \ne 0$, ``the points at infinity'', i.e. the points in the complement
$N^{p,q} \setminus \iota(\mathbb{R}^{p,q})$, can be put in bijection with the
points in the null cone
$\{\overrightarrow{x}\in\mathbb{R}^{p,q} ;\: Q(\overrightarrow{x})=0\}$. 

The indefinite orthogonal group $O(p+1,q+1)$ acts on $\mathbb{R}^{p+1,q+1}$
by linear transformations. This action descends to the projective space
$\mathbb{R}P^{p+q+1}$ and preserves $N^{p,q}$.
Thus, for every element $M \in O(p+1,q+1)$, we obtain a transformation
on $N^{p,q}$, which will be denoted by $\psi_M$.

\begin{thm}
For each $M \in O(p+1,q+1)$, the map $\psi_M: N^{p,q} \to N^{p,q}$
is a conformal transformation and a diffeomorphism.
If $M, M' \in O(p+1,q+1)$, then $\psi_M = \psi_{M'}$ if and only if $M' = \pm M$.
\end{thm}

The following proposition essentially asserts that the restrictions of $\psi_M$
to $\mathbb{R}^{p,q}$ are compositions of parallel translations, rotations,
dilations and inversions.

\begin{prop}\label{types}
For each $M \in O(p+1,q+1)$, the conformal transformation $\psi_M$ can be
written as a composition of $\psi_{M_j}$, for some $M_j \in O(p+1,q+1)$,
$j=1,\dots,n$, such that each $\iota^{-1} \circ \psi_{M_j}$ is one of the
following types of transformations:
\begin{enumerate}
\item parallel translations
$\overrightarrow{x} \mapsto \overrightarrow{x}+\overrightarrow{b}$, where
$\overrightarrow{b} \in \mathbb{R}^{p,q}$;
\item linear transformations $\overrightarrow{x}\mapsto M\overrightarrow{x}$,
where $M \in O(p,q)$;
\item dilations $\overrightarrow{x} \mapsto \lambda \overrightarrow{x}$,
where $\lambda>0$;
\item an inversion
$\inv(\overrightarrow{x})=\frac{\overrightarrow{x}}{Q(\overrightarrow{x})}$.
\end{enumerate}
\end{prop}

For future reference we observe that a diagonal matrix
\begin{equation}  \label{omega}
\omega_0 =
\begin{pmatrix}-1 & & & \\ & 1 & & \\ & & \ddots & \\ & & & 1 \end{pmatrix}
\in O(p+1,q+1)
\end{equation}
and its negative are the two matrices in $O(p+1,q+1)$ such that
$$
\iota^{-1} \circ \psi_{\omega_0} \circ \iota
= \iota^{-1} \circ \psi_{-\omega_0} \circ \iota = \inv:
\overrightarrow{x} \mapsto \frac{\overrightarrow{x}}{Q(\overrightarrow{x})}.
$$
While $\inv$ is defined only for
$\{ \overrightarrow{x} \in \mathbb{R}^{p,q} ;\: Q(\overrightarrow{x}) \ne 0 \}$,
the conformal transformation $\psi_{\omega_0} = \psi_{-\omega_0}$ is
well-defined on the entire $N^{p,q}$.

When $p+q>2$, these are {\em all} possible conformal transformations of
$N^{p,q}$. In fact, an even stronger result is true.

\begin{thm}
Let $p+q>2$. Every conformal transformation $\phi :U \to \mathbb{R}^{p,q}$,
where $U \subset \mathbb{R}^{p,q}$ is any non-empty connected open subset,
can be uniquely extended to $N^{p,q}$, i.e. there exists a unique conformal
diffeomorphism $\hat{\phi} : N^{p,q}\to N^{p,q}$ such that
$\hat{\phi}(\iota(\overrightarrow{x}))=\iota(\phi(\overrightarrow{x}))$
for all $\overrightarrow{x} \in U$.

Moreover, every conformal diffeomorphism $N^{p,q}\to N^{p,q}$
must be of the form $\psi_M$, for some $M \in O(p+1,q+1)$.
Thus, the group of all conformal transformations $N^{p,q}\to N^{p,q}$
is isomorphic to $O(p+1,q+1)/\{\pm Id\}$.
\end{thm}


\subsection{Example: The $\mathbb{R}^{2,0} \simeq \mathbb{C}$ Case}  \label{C}

The Euclidean plane $\mathbb{R}^{2,0}$ can be identified with the complex
plane $\mathbb{C}$.
In this case, the conformal compactification $N^{2,0}$ is just the
one-point compactification $\hat{\mathbb{C}} = \mathbb{C} \cup \{\infty\}$.
The orientation-preserving conformal transformations $\phi : U \to \mathbb{C}$,
where $U \subset \mathbb{C}$ is an open subset, are exactly those
holomorphic functions with nowhere-vanishing derivative.
Thus, it is no longer true that every conformal transformation on $\mathbb{C}$
is obtained by restricting some $\psi_M$, $M \in O(3,1)$.

\begin{prop}
Every conformal diffeomorphism $N^{2,0} \to N^{2,0}$ must be
of the form $\psi_M$, for some $M \in O(3,1)$.
\end{prop}

Recall from elementary complex analysis (see, for example, \cite{A}) that
we have an action of $SL(2,\mathbb{R})$ on $\hat{\mathbb{C}} \simeq N^{2,0}$
by M\"obius (or fractional linear) transformations:
$$
z \mapsto \frac{az+b}{cz+d}, \qquad
\begin{pmatrix}a & b \\ c & d\end{pmatrix} \in SL(2,\mathbb{C}).
$$
Each of these transformations is conformal and orientation-preserving.
These are precisely the transformations produced by $SO(3,1) \subset O(3,1)$.
The other transformations produced by $O(3,1)$ are orientation-reversing
and can be described as
$$
z \mapsto \frac{a \bar z +b}{c \bar z +d}, \qquad
\begin{pmatrix}a & b \\ c & d\end{pmatrix} \in SL(2,\mathbb{C}),
$$
where $\bar z$ denotes the complex conjugate of $z$.
These are all the conformal diffeomorphisms
$\hat{\mathbb{C}} \to \hat{\mathbb{C}}$.

We are particularly interested in the following properties of the M\"obius
transformations:

\begin{prop}  \label{circles}
The M\"obius transformations on $\hat{\mathbb{C}}$ map circles and lines
into circles and lines.
Moreover, this action of $SL(2,\mathbb{C})$ on circles and lines is transitive,
i.e. every circle or line in $\hat{\mathbb{C}}$ can be mapped into any
other circle or line by a M\"obius transformation.
\end{prop}

\subsection{Example: The $\mathbb{R}^{1,1}$ Case}  \label{R^{1,1}}

It is possible to describe all conformal transformations of $\mathbb{R}^{1,1}$
as well as all conformal diffeomorphisms $N^{1,1} \to N^{1,1}$
(see, for example, \cite{SCH}).
Not every conformal transformation of $\mathbb{R}^{1,1}$ can be extended to
$N^{1,1}$; and not every conformal diffeomorphism $N^{1,1} \to N^{1,1}$ is of
the form $\psi_M$, for some $M \in O(2,2)$.

A result stated in \cite{YAG} after Equation (49') in Supplement C can be
regarded as an analogue of Proposition \ref{circles} and can be loosely
paraphrased as ``the elements of $O(2,2)$ map hyperbolas in $\mathbb{R}^{1,1}$
with asymptotes parallel to $x_2=\pm x_1$ and lines into hyperbolas of
the same kind and lines''.

\subsection{Example: $N^{2,2}$ Is the Grassmannian $Gr(2, \mathbb{R}^4)$}

We show that the conformal compactification $N^{2,2}$ together with the action
of $SO(3,3)$ is the same as the Grassmannian of $2$-planes in $\mathbb{R}^4$,
$Gr(2, \mathbb{R}^4)$, together with the natural action of $SL(4,\mathbb{R})$.
This example illustrates the following observation made in Subsection 2.7
of Part III of \cite{KO}: The conformal compactification $N^{p,q}$ can be
regarded as the generalized real flag variety $G/P^{\text{max}}$ with
$\mathbb{R}^{p,q}$ being the open Bruhat cell.
The space $Gr(2, \mathbb{R}^4)$ plays a central role in the split real form
of the Penrose transform \cite{M}.

We essentially follow the proof that $SL(4,\mathbb{R})$ is locally isomorphic
to $SO(3,3)$, as given in \cite{H} (see also \cite[Theorem 1.3.1]{WW}).
Pick a basis $\{\overrightarrow{e_1},\overrightarrow{e_2},\overrightarrow{e_3},
\overrightarrow{e_4}\}$ of $\mathbb{R}^4$. We define a symmetric bilinear form
$B : \bigwedge^2 \mathbb{R}^4 \times \bigwedge^2 \mathbb{R}^4 \to \mathbb{R}$
by the condition that
$$
v\wedge w = B(v,w) \cdot \overrightarrow{e_1} \wedge \overrightarrow{e_2}
\wedge \overrightarrow{e_3} \wedge \overrightarrow{e_4},
\qquad v,w \in \bigwedge^2 \mathbb{R}^4.
$$
Let $Q(v)=B(v,v)$ be the corresponding quadratic form.
Define an isomorphism $\bigwedge^2 \mathbb{R}^4 \simeq \mathbb{R}^6$
by specifying a basis
\begin{equation}  \label{basis}
f_1 =\overrightarrow{e_1} \wedge \overrightarrow{e_2}, \quad
f_2 = \overrightarrow{e_1} \wedge \overrightarrow{e_3}, \quad
f_3 = \overrightarrow{e_1} \wedge \overrightarrow{e_4}, \quad
f_4 = \overrightarrow{e_2} \wedge \overrightarrow{e_3}, \quad
f_5 = \overrightarrow{e_2} \wedge \overrightarrow{e_4}, \quad
f_6 = \overrightarrow{e_3} \wedge \overrightarrow{e_4}.
\end{equation}
Then
$$
Q(v)=2x_1x_6-2x_2x_5+2x_3x_4, \qquad v \in \bigwedge^2 \mathbb{R}^4,
$$
where the $x_j$'s are the coordinates of $v$ with respect to
the basis (\ref{basis}).
This shows that $Q$ is nondegenerate and has signature $(3,3)$.
Moreover, since the form $B$ is $SL(4,\mathbb{R})$-invariant,
the linear transformations on $\bigwedge^2 \mathbb{R}^4$ by elements of
$SL(4,\mathbb{R})$ result in transformations by elements of $O(3,3)$ on
$\mathbb{R}^6$. This establishes an isomorphism of
$SL(4,\mathbb{R}) / \{\pm Id\}$ with the connected component of
the identity element of $O(3,3)$.

Clearly, every {\em decomposable} $v \in \bigwedge^2 \mathbb{R}^4$, i.e.
every element that can be written as
$v = \overrightarrow{x} \wedge \overrightarrow{y}$
for some $\overrightarrow{x}, \overrightarrow{y} \in \mathbb{R}^4$,
satisfies $v \wedge v =0$.
Conversely, if $v \in \bigwedge^2 \mathbb{R}^4$ has the property that
$v \wedge v =0$, then it must be decomposable (c.f. \cite[Lemma 1.3.2]{WW}).
Therefore, the set of decomposable elements in $\bigwedge^2\mathbb{R}^4$ can be
described as the cone
$$
\Bigl\{ v \in \bigwedge^2\mathbb{R}^4 ;\: v \wedge v =0 \Bigr\}
= \Bigl\{ v \in \bigwedge^2\mathbb{R}^4 ;\: Q(v) =0 \Bigr\}.
$$

Finally, the Pl\"{u}cker embedding
$Gr(2,\mathbb{R}^4) \hookrightarrow \mathbb{P}(\bigwedge^2\mathbb{R}^4)$
gives an identification between the Grassmannian of $2$-planes in
$\mathbb{R}^4$ and the equivalence classes of decomposable elements in
$\bigwedge^2\mathbb{R}^4 \setminus \{0\}$.
Thus we have a natural diffeomorphism $Gr(2,\mathbb{R}^4) \simeq N^{2,2}$
respecting the group actions and commuting with the group isomorphism
$SL(4,\mathbb{R}) / \{\pm Id\} \simeq SO(3,3) / \{\pm Id\}$.

\section{Hyperboloids and Hyperplanes in $\mathbb{R}^{p,q}$}  \label{S3}

This paper is mainly concerned with geometric properties of the action of
$O(p+1,q+1)$ on $\mathbb{R}^{p,q}$ that generalize the
``mapping circles and lines onto circles and lines'' properties stated in
Proposition \ref{circles} for the case of $\mathbb{R}^{2,0} \simeq \mathbb{C}$.
Note that when switching from the positive definite case of
$\mathbb{R}^{2,0}$ (Subsection \ref{C}) to the indefinite case of
$\mathbb{R}^{1,1}$ (Subsection \ref{R^{1,1}}), circles were replaced by
hyperbolas with asymptotes parallel to $x_2=\pm x_1$.
Observe further, that in both cases, the curves in question are simply the
zero loci of the quadrics $Q(x)=c$, for $c\in\mathbb{R}$, up to translation.
In the general case, where the quadratic form $Q$ is not necessarily
positive definite, circles will be replaced with hyperboloids.
Therefore, we wish to define a suitable space of hyperboloids and hyperplanes.

As a first try, let ${\cal H}$ be the set of quadratic hypersurfaces in
$\mathbb{R}^{p,q}$ that can be put in the form
$$
\bigl\{ \overrightarrow{x} \in \mathbb{R}^{p,q} ;\:
\alpha Q(\overrightarrow{x}) + B(\beta,\overrightarrow{x})+\gamma=0 \bigr\}
$$
for some choice of parameters $\alpha,\gamma \in \mathbb{R}$ and
$\beta\in\mathbb{R}^{p,q}$ subject to the requirement that this set of
solutions is indeed a hypersurface, i.e. has codimension $1$ in
$\mathbb{R}^{p,q}$.
For example, when $\alpha=0$ we obtain affine hyperplanes
(i.e. hyperplanes not necessarily passing through the origin).
If $\alpha \ne 0$ and $p$ or $q$ is $0$ we obtain spheres.
In the most typical case $\alpha \ne 0$ and $p,q \ne 0$, these hypersurfaces
are either hyperboloids or cones.

We would like to prove a statement to the effect that the conformal action of
$O(p+1,q+1)$ restricted to $\mathbb{R}^{p,q}$ preserves the quadratic
hypersurfaces in ${\cal H}$.
The following argument can be made precise if $p$ or $q$ is $0$ and is similar
to the one given in \cite{BG}. However, it fails in general.
We start with a quadratic hypersurface $H_0 \in {\cal H}$ given by an equation
$$
\alpha_0 Q(\overrightarrow{x}) + B(\beta_0,\overrightarrow{x}) + \gamma_0 =0,
$$
where $\alpha_0,\gamma_0\in\mathbb{R}$ and $\beta_0\in\mathbb{R}^{p,q}$.
By Proposition \ref{types}, it is sufficient to prove that
parallel translations, rotations by elements of $O(p,q)$, dilations
and the inversion $\inv$ send $H_0$ into another element of ${\cal H}$.
There are no problems with translations, rotations and dilations.
Applying the map $\inv$ to $H_0$, we have:
\begin{align*}
0 &= \alpha_0 Q\left (\frac{\overrightarrow{x}}{Q(\overrightarrow{x})}\right)
+ B \left( \beta_0,\frac{\overrightarrow{x}}{Q(\overrightarrow{x})} \right)
+ \gamma_0 \\
&= \frac{Q(\overrightarrow{x})}{Q(\overrightarrow{x})^2}\alpha_0
+ \frac1{Q(\overrightarrow{x})} B(\beta_0,\overrightarrow{x})+\gamma_0 \\
&= \frac1{Q(\overrightarrow{x})} \alpha_0
+ \frac1{Q(\overrightarrow{x})}B(\beta_0,\overrightarrow{x})+\gamma_0 \\
&= \alpha_0+B(\beta_0,\overrightarrow{x})+\gamma_0 Q(\overrightarrow{x}).
\end{align*}

\begin{rem}
At this point we can see why it is necessary to consider hypersurfaces
with quadratic terms proportional to $Q(\overrightarrow{x})$ only.
Otherwise, applying the inversion map would result in an equation involving
terms of degree four.
\end{rem}

The problem with this argument is that, when $p,q \ne 0$,
the hypersurface $H_0$ may have a significant intersection with the null cone
$\{ Q(\overrightarrow{x}) =0 \}$. In fact, as the next example shows,
$H_0$ may lie entirely inside the null cone.

\begin{ex}  \label{R^{1,1}-ex}
Let $p=q=1$, and consider two lines $L_+$ and $L_-$ in $\mathbb{R}^{1,1}$
given by equations $x_1+x_2=0$ and $x_1-x_2=0$ respectively.
Both lines lie inside the null cone $\{ Q(\overrightarrow{x}) =0 \}$,
where the inversion map $\inv$ is not defined.
Recall the element $\omega_0$ defined by (\ref{omega}), which generates $\inv$.
The map $\psi_{\omega_0}: N^{1,1} \to N^{1,1}$ maps the lines $L_+$ and $L_-$ into
``infinity'', i.e. the complement of $\iota(\mathbb{R}^{1,1})$ in $N^{1,1}$.
\end{ex}

This example demonstrates that it is necessary to work in the conformal
compactification $N^{p,q}$ rather than $\mathbb{R}^{p,q}$ and that it is also
necessary to have a suitable alternative to ${\cal H}$ consisting of
quadratic hypersurfaces in $N^{p,q}$.
The most na\"ive notion of hypersurfaces in $N^{p,q}$ defined as the closures
of quadratic hypersurfaces from ${\cal H}$ does not work.

\section{Conformal Quadratic Hypersurfaces in $N^{p,q}$}  \label{S4}

In this section we define the analogues of hyperboloids and affine hyperplanes
in the conformal compactification $N^{p,q}$ on which the action of $O(p+1,q+1)$
is well-defined.

Let $\hat B$ denote the bilinear form on $\mathbb{R}^{p+1,q+1}$ such that
$\hat B(\overrightarrow{\xi},\overrightarrow{\xi})
= \hat Q(\overrightarrow{\xi})$
(recall that the quadratic form $\hat Q$ is defined by (\ref{Q-hat})).
Then every hyperplane $h$ in $\mathbb{R}^{p+1,q+1}$ or $\mathbb{R}P^{p+q+1}$
can be expressed as
$$
h = \bigl\{ \overrightarrow{\xi} \in \mathbb{R}^{p+1,q+1} ;\:
\hat B(\overrightarrow{a}, \overrightarrow{\xi}) =0 \bigr\}
$$
for some non-zero vector $\overrightarrow{a} \in \mathbb{R}^{p+1,q+1}$.
The vector $\overrightarrow{a}$ is uniquely determined by $h$ up to scaling.
The following lemma shows that every element of ${\cal H}$ can be realized
as an intersection of $N^{p,q}$ with a hyperplane in $\mathbb{R}P^{p+q+1}$.

\begin{lem}
Every quadratic hypersurface $H \in {\cal H}$ can be realized as the pullback
of the intersection of $N^{p,q}$ with a unique hyperplane in $\mathbb{R}P^{p+q+1}$.
\end{lem}

\begin{proof}
Let
$$
\overrightarrow{a} = (a_0,a_1,\dots,a_{p+q},a_{p+q+1}) \in \mathbb{R}^{p+1,q+1}
$$
be a non-zero vector.
The corresponding hyperplane in $\mathbb{R}P^{p+q+1}$ is
$$
h = \bigl\{ [\overrightarrow{\xi}] \in \mathbb{R}P^{p+q+1} ;\:
a_0\xi_0 + \dots + a_p\xi_p - a_{p+1}\xi_{p+1} - \dots - a_{p+q+1}\xi_{p+q+1} =0
\bigr\}.
$$
Then the pullback $\iota^{-1} (N^{p,q} \cap h)$ is given by an equation
\begin{align*}
0 &= a_0 \frac{1-Q(\overrightarrow{x})}2 + a_1x_1 + \dots + a_px_p
-a_{p+1}x_{p+1} - \dots - a_{p+q}x_{p+q} - a_{p+q+1}\frac{1+Q(\overrightarrow{x})}2 \\
&= \alpha_0 Q(\overrightarrow{x}) + B(\beta_0,\overrightarrow{x}) + \gamma_0,
\end{align*}
where
$$
\alpha_0 = -\frac12(a_0+a_{p+q+1}), \qquad
\beta_0 = (a_1,\dots,a_{p+q}), \qquad
\gamma_0 = \frac12(a_0 - a_{p+q+1}).
$$
This proves that every $H \in {\cal H}$ can be realized as
$\iota^{-1} (N^{p,q} \cap h)$ and that the hyperplane $h$ in
$\mathbb{R}P^{p+q+1}$ is uniquely determined by the hypersurface $H$.
\end{proof}

This lemma suggests that the analogues of hyperboloids and affine hyperplanes
in $N^{p,q}$ should be defined as intersections with hyperplanes in
$\mathbb{R}P^{p+q+1}$.
However, in some cases, these intersections may have codimension strictly
greater than $1$. We now turn our attention to such cases.

\begin{df}
Let $h$ be a hyperplane in $\mathbb{R}P^{p+q+1}$ given by the equation
$\hat{B}(\overrightarrow{a},\overrightarrow{\xi})=0$ for some
non-zero $\overrightarrow{a} \in \mathbb{R}^{p+1,q+1}$.
The {\em sign of $h$} is the sign of $\hat{Q}(\overrightarrow{a})$;
it can be positive, negative or zero.
\end{df}

Since $\overrightarrow{a}$ is uniquely determined by $h$ up to
proportionality, this notion makes sense.

\begin{lem}  \label{intersection-lem}
Let $h$ be a hyperplane in $\mathbb{R}P^{p+q+1}$. If $p, q \ne 0$, then
the intersection $N^{p,q} \cap h$ always has codimension $1$ in $N^{p,q}$.

If $p$ or $q$ is $0$, the dimension of $N^{p,q} \cap h$ is completely determined
by the sign of $h$:
\begin{enumerate}
\item
The intersection has codimension $1$ in $N^{p,q}$ if and only if
$q=0$ and $h$ has positive sign or $p=0$ and $h$ has negative sign;
\item
The intersection consists of a single point if and only if
$h$ has sign zero;
\item
The intersection is empty if and only if
$q=0$ and $h$ has negative sign or $p=0$ and $h$ has positive sign.
\end{enumerate}
\end{lem}

\begin{proof}
The dimension of $N^{p,q} \cap h$ can be computed by finding
the codimension of the intersection in the vector space underlying $h$.
If $p, q \ne 0$, then the restriction of $\hat{Q}$ to $h$ is indefinite
and the codimension is $1$. This leaves us with the cases $p$ or $q$ is $0$.

For concreteness, suppose that $q=0$ (the case $p=0$ follows similarly).
Let the equation of $h$ be
$\hat{B}(\overrightarrow{a},\overrightarrow{\xi})=0$.
We can rotate the hyperplane by an element of $O(p+1) \subset O(p+1,1)$
so that all coordinates of $\overrightarrow{a}$ are zero,
except possibly the first and last ones.
Doing so leaves the dimension of the intersection as well as the sign of $h$
invariant. Therefore, we may view $h$ as being given by the equation
$a_0\xi_0+a_{p+1}\xi_{p+1}=0$. 

Suppose first that $a_{p+1}=0$. Then $h$ is given by $\xi_0=0$ and
$N^{p,0} \cap h $ is given by
$$
(\xi_1)^2+\dots+(\xi_p)^2-(\xi_{p+1})^2 =0,
\qquad \xi_0=0.
$$
This intersection has codimension $1$ in $N^{p,q}$ and
$\hat{Q}(\overrightarrow{a}) >0$.

Next, suppose that $a_{p+1} \ne 0$.
Then the equation defining $h$ can be rewritten as
$\xi_{p+q+1}=-\frac{a_0}{a_{p+q+1}}\xi_0$ and $N^{p,0} \cap h$ is given by 
$$
\left( 1- \frac{(a_0)^2}{(a_{p+1})^2} \right)(\xi_0)^2 
+(\xi_1)^2+\dots+(\xi_p)^2 =0,  \qquad
\xi_{p+1} =-\frac{a_0}{a_{p+1}}\xi_0.
$$
We have three cases:
\begin{enumerate}
\item
If $1-\frac{(a_0)^2}{(a_{p+1})^2} <0$, then $N^{p,0} \cap h$ has codimension $1$
in $N^{p,0}$ and $\hat{Q}(\overrightarrow{a}) >0$;
\item
If $1-\frac{(a_0)^2}{(a_{p+1})^2} >0$, then the intersection is empty in $N^{p,0}$
and $\hat{Q}(\overrightarrow{a}) <0$;
\item
If $1-\frac{(a_0)^2}{(a_{p+1})^2}=0$, then $a_0=\pm a_{p+1}$,
the intersection consists of a single point
$[1:0:\dots:0:-1]$ or $[1:0:\dots:0:1]$ in $\mathbb{R}P^{p+1}$
and $\hat{Q}(\overrightarrow{a}) =0$.
\end{enumerate}
\end{proof}

With this lemma in mind, we introduce our key definition:

\begin{df}  \label{def}
A {\em conformal quadratic hypersurface $H$ in $N^{p,q}$} is the intersection
of a hyperplane $h \subset \mathbb{R}P^{p+q+1}$ with $N^{p,q}$.
Additionally, if $q=0$ (or $p=0$), we require the hyperplane $h$
to have positive (respectively negative) sign.

We denote by $\hat{\cal H}$ the set of all conformal quadratic hypersurfaces
in $N^{p,q}$, and define the {\em sign of $H \in \hat{\cal H}$} to be
the sign of the corresponding hyperplane $h$.
\end{df}

It is easy to verify that the hyperplane $h \subset \mathbb{R}P^{p+q+1}$
is uniquely determined by a conformal quadratic hypersurface
$H \in \hat{\cal H}$. In particular, the sign of $H$ is well-defined.
When $p$ or $q$ is $0$, these conformal quadratic hypersurfaces are precisely
the closures in $N^{p,q}$ of spheres and affine planes in $\mathbb{R}^{p+q}$ of
dimension $p+q-1$.

\section{The Group $O(p+1,q+1)$ Acting on the Conformal Quadratic Surfaces in
$N^{p,q}$}  \label{S5}

In this section we prove our main results of this paper.
For convenience we state the following result (see, for example, \cite{LAM}).

\begin{lem}[Witt's Extension Theorem]  \label{Witt}
Let $V$ be a finite-dimensional vector space (over $\mathbb{R}$)
together with a nondegenerate symmetric bilinear form $B$.
If $\phi :W_1 \to W_2$ is an isometric isomorphism of two subspaces
$W_1, W_2 \subset V$, then $\phi$ extends to an isometric isomorphism
$\hat{\phi} : V \to V$.
\end{lem}

The first version of our main result can be stated as follows:

\begin{thm}  \label{main-thm}
The group $O(p+1,q+1)$ preserves the set $\hat{\cal H}$.
Furthermore, two elements of $\hat{\cal H}$ can be mapped one into the other
by the action of $O(p+1,q+1)$ if and only if they have the same sign.
In particular, if $p$ or $q$ is $0$ the action of $O(p+1,q+1)$ on $\hat{\cal H}$
is transitive; if $p, q \ne 0$, this action has exactly three orbits.
\end{thm}

\begin{proof}
By Definition \ref{def}, an element of $\hat{\cal H}$ is the intersection of
$N^{p,q}$ with a hyperplane in $\mathbb{R}P^{p+q+1}$.
Since $O(p+1,q+1)$ acts linearly on $\mathbb{R}P^{p+q+1}$ and maps $N^{p,q}$ into
itself, $\hat{\cal H}$ is certainly preserved by this action.

For the second part, let $H,H' \in \hat{\cal H}$ be arbitrary.
Then $H = N^{p,q} \cap h$ and $H' = N^{p,q} \cap h'$,
where $h$ and $h'$ are hyperplanes in $\mathbb{R}P^{p+q+1}$ given by
equations $\hat{B}(\overrightarrow{a},\overrightarrow{\xi})=0$ and
$\hat{B}(\overrightarrow{a}',\overrightarrow{\xi})=0$ respectively.
If there exists an $M \in O(p+1,q+1)$ mapping $H$ into $H'$, then
$M$ must map $h$ into $h'$, and $M \overrightarrow{a}$ must be proportional
to $\overrightarrow{a}'$.
That is, $M \overrightarrow{a} = \lambda \overrightarrow{a}'$,
for some $\lambda \ne 0$. Hence
$$
\hat{Q}(\overrightarrow{a})
= \hat{Q}(M \overrightarrow{a})
= \hat{Q}(\lambda \overrightarrow{a}')
= \lambda^2 \hat{Q}(\overrightarrow{a}').
$$
Thus, $\hat{Q}(\overrightarrow{a})$ and $\hat{Q}(\overrightarrow{a}')$
have the same sign, which by definition means that $H$ and $H'$
have the same sign.

Conversely, suppose that $H$ and $H'$ have the same sign.
In other words, $\hat{Q}(\overrightarrow{a})$ and
$\hat{Q}(\overrightarrow{a}')$ have the same sign.
Therefore, there exists some $\lambda \ne 0$ such that
\begin{equation}  \label{equal}
\hat{Q}(\overrightarrow{a})
= \lambda^2 \cdot \hat{Q}(\overrightarrow{a}')
= \hat{Q}(\lambda \overrightarrow{a}').
\end{equation}
Let $m : \mathbb{R} \overrightarrow{a} \to \mathbb{R} \overrightarrow{a}'$
be a linear transformation determined by
$\overrightarrow{a} \mapsto \lambda \overrightarrow{a}'$;
by (\ref{equal}), this is an isometric isomorphism between two
$1$-dimensional subspaces of $\mathbb{R}^{p+1,q+1}$.
By Witt's Extension Theorem (Lemma \ref{Witt}), there exists an
$M \in O(p+1,q+1)$ such that
$M \overrightarrow{a} =\lambda \overrightarrow{a}'$.
Since $M$ preserves $\hat{B}$, $M(h)= h'$;
that is, $M \cdot H = H'$, as desired.
\end{proof}

\begin{ex}
Let $p=q=1$. Consider three curves in $\mathbb R^{1,1}$:
\begin{equation}  \label{curves}
(x_1)^2-(x_2)^2=1, \qquad
(x_1)^2-(x_2)^2=-1, \qquad
(x_1)^2-(x_2)^2=0.
\end{equation}
These curves can be realized as intersections of $N^{1,1}$ with hyperplanes in
$\mathbb RP^3$ given by equations
$$
\xi_0=0, \qquad
\xi_3=0, \qquad
\xi_0-\xi_3=0
$$
respectively. Then vectors orthogonal to these hyperplanes are
$$
\overrightarrow{a}_1 = (1,0,0,0), \qquad
\overrightarrow{a}_2 = (0,0,0,1), \qquad
\overrightarrow{a}_3 = (1,0,0,1).
$$
Hence the signs of the three curves (\ref{curves}) are positive, negative and
zero respectively.
From Theorem \ref{main-thm} we conclude that there is no M\"{o}bius
transformation on $\mathbb R^{1,1}$ mapping one of these curves into another.
Moreover, any conformal quadratic curve in $N^{1,1}$ can be mapped
into exactly one of these curves by a M\"{o}bius transformation.
\end{ex}

When $p$ or $q$ is $0$, $N^{p,q}$ is the one-point compactification of
$\mathbb{R}^{p+q}$ and $\hat{\cal H}$ is the set of hyperspheres and
closures of affine hyperplanes in $\mathbb{R}^{p+q}$.

\begin{lem}  \label{codim1}
Let $p$ or $q$ be $0$ and $d=1,\dots, p+q-1$. Then every $d$-dimensional
sphere or affine plane in $\mathbb{R}^{p+q}$ can be realized as
an intersection of $N^{p,q}$ with a $(d+1)$-dimensional projective subspace
$h \subset \mathbb{R}P^{p+q+1}$ such that the restriction of $\hat{Q}$ to
the underlying vector space of $h$ is nondegenerate and indefinite.
\end{lem}

\begin{proof}
Note that every $d$-dimensional sphere or affine plane can be realized as an
intersection of two $(d+1)$-dimensional spheres or affine planes.
Hence, by induction on codimension, every $d$-dimensional sphere or affine
plane can be realized as an intersection of $N^{p,q}$ with a projective
subspace $h$ of dimension $d+1$ in $\mathbb{R}P^{p+q+1}$.

If $h$ is a $(d+1)$-dimensional projective subspace of $\mathbb{R}P^{p+q+1}$,
the dimension of $N^{p,q} \cap h$ can be computed by finding
the codimension of the intersection in the vector space underlying $h$.
If $\hat{Q}|_h$ is indefinite, the codimension is $1$;
if $\hat{Q}|_h$ is definite, $N^{p,q} \cap h = \varnothing$;
and if $\hat{Q}|_h$ is semidefinite, the codimension is the rank of
$\hat{Q}|_h$.

Now, when $p$ or $q$ is $0$, if $\hat{Q}|_h$ is semidefinite, then it must
be nondegenerate, and the case $\hat{Q}|_h$ has rank $1$ cannot happen.
\end{proof}

We can extend Theorem \ref{main-thm} to spheres and affine planes of lower
dimension.

\begin{thm}  \label{main-thm2}
When $p$ or $q$ is $0$, the group $O(p+1,q+1)$ acts transitively on the
set of $d$-dimensional spheres and affine planes in $\mathbb{R}^{p+q}$,
where $d=1,\dots, p+q-1$.
\end{thm}

\begin{proof}
The group $O(p+1,q+1)$ acts on $\mathbb{R}P^{p+q+1}$ linearly and maps
$N^{p,q}$ into itself, hence maps the intersections $N^{p,q} \cap h$,
where $h$ is a projective subspace of dimension $d+1$,
into intersections of the same kind preserving the signature of $\hat{Q}|_h$.
This proves that $O(p+1,q+1)$ maps $d$-dimensional spheres and affine planes
into spheres and affine planes of the same dimension.

To prove the transitivity of the action, consider two spheres or affine planes
of dimension $d$ realized as intersections $N^{p,q} \cap h$ and $N^{p,q} \cap h'$
respectively.
Since the restrictions $\hat{Q}|_h$ and $\hat{Q}|_{h'}$ have the same signature,
there is an isometric isomorphism $m: h \to h'$ of the underlying vector spaces.
By Witt's Extension Theorem (Lemma \ref{Witt}), there exists an
$M \in O(p+1,q+1)$ extending this isomorphism.
This $M$ maps the first sphere or affine plane into the second.
\end{proof}


Generalizing the above results to arbitrary $\mathbb{R}^{p,q}$ requires
an analogue of lower dimensional hyperboloids and affine planes in $N^{p,q}$.
Inspired by the proof of Theorem \ref{main-thm2}, we can define a
{\em $d$-dimensional conformal quadratic surface in $N^{p,q}$} as the
intersection of $N^{p,q}$ with a projective subspace $h$ of dimension $d+1$ in
$\mathbb{R}P^{p+q+1}$, provided that this intersection indeed has dimension $d$,
where $d=1,\dots,p+q-1$.
As an immediate consequence of the proof of Lemma \ref{codim1}, we obtain:

\begin{lem}
Let $h$ be a projective subspace of $\mathbb{R}P^{p+q+1}$ of dimension $d+1$.
Then the intersection $N^{p,q} \cap h$ has dimension $d$ if and only if the
restriction of $\hat{Q}$ to the vector space underlying $h$
is indefinite (degenerate or nondegenerate) or has rank $1$.
\end{lem}



We denote the set of $d$-dimensional conformal quadratic surfaces in $N^{p,q}$
by $\hat{\cal H}_d$ (then $\hat{\cal H}_{p+q-1} = \hat{\cal H}$).
As the following example shows, the $(d+1)$-dimensional projective subspace
$h$ is no longer uniquely determined by $H \in \hat{\cal H}_d$.

\begin{ex}
Consider $N^{2,2} \subset \mathbb{R}P^5$ and $d=1$.
Choose four vectors in $\mathbb{R}^{3,3}$:
$$
\overrightarrow{v_1} = (1,0,0,0,0,0), \quad
\overrightarrow{v_2} = (0,1,0,0,1,0), \quad
\overrightarrow{v_3} = (0,0,1,1,0,0), \quad
\overrightarrow{v_4} = (0,0,0,0,0,1),
$$
and let
$$
V_1 = \operatorname{Span} \{ \overrightarrow{v_1},
\overrightarrow{v_2}, \overrightarrow{v_3} \}, \qquad
V_2 = \operatorname{Span} \{ \overrightarrow{v_2},
\overrightarrow{v_3}, \overrightarrow{v_4} \}, \qquad
W = V_1 \cap V_2 = \operatorname{Span} \{ \overrightarrow{v_2},
\overrightarrow{v_3} \}.
$$
Then $\hat{Q}|_{V_1}$ is positive semidefinite of rank $1$,
$\hat{Q}|_{V_2}$ is negative semidefinite of rank $1$ and
$$
N^{2,2} \cap [V_1] = N^{2,2} \cap [V_2] = [W] \subset \mathbb{R}P^5,
$$
i.e. the conformal quadratic curve $[W] \in \hat{\cal H}_1$ can be realized
as an intersection $N^{2,2} \cap h$ with different -- in fact infinitely many --
projective subspaces $h \subset \mathbb{R}P^5$ of dimension $2$.
Moreover, the restrictions $\hat{Q}|_h$ may have different signatures.
\end{ex}

Keeping in mind this example, we can state the following generalization of
Theorems \ref{main-thm} and \ref{main-thm2};
its proof is the same as that of Theorem \ref{main-thm2}.

\begin{thm}  \label{main-thm3}
The action of $O(p+1,q+1)$ on $N^{p,q}$ preserves each $\hat{\cal H}_d$,
where $d=1,\dots,p+q-1$.
Furthermore, two conformal quadratic surfaces $H$ and $H'$ in $\hat{\cal H}_d$
are in the same orbit if and only if they can be realized as intersections
$N^{p,q} \cap h$ and $N^{p,q} \cap h'$ respectively with
restrictions $\hat{Q}|_h$ and $\hat{Q}|_{h'}$ having the same signature.
\end{thm}

\end{document}